\documentclass[10pt]{amsart}
\usepackage{amsmath,amstext,amssymb,amsopn,amsthm,mathrsfs}
\usepackage{graphicx}
\usepackage[english]{babel}
\usepackage{times}
\usepackage{fancybox}
\usepackage{epsfig}
\usepackage{graphics}
\usepackage{color}
\usepackage{amsmath}
\usepackage{amsfonts}
\usepackage[ansinew]{inputenc}
\usepackage{tikz}

\newtheorem{theorem}{Theorem}[section]
\newtheorem{defi}{Definition}[section]

\numberwithin{equation}{section}

\begin{document}

\title[On Cantor-like sets] {On Cantor-like sets and Cantor-Lebesgue singular functions}

\author{Robert DiMartino}
\address{Department of Mathematical and Actuarial Sciences, Roosevelt University  Chicago, Il, 60605, USA.}
\email{[Robert DiMartino]robert.dimartino@gmail.com}
\author{Wilfredo O. Urbina}
\email{[Wilfredo Urbina]wurbinaromero@roosevelt.edu}

\begin{abstract}
In this paper we discuss several variations and generalizations of the Cantor set and study some of their properties. Also for each of those generalizations a Cantor-like function can be constructed from the set. We will discuss briefly the possible construction of those functions.
\end{abstract}
\maketitle

\section{Introduction.}
The Cantor  ternary set is the best example of a {\em perfect nowhere-dense} set in the real line. It was constructed by George Cantor in 1883, see \cite{cantor}, nevertheless it was not the first perfect nowhere-dense set  in the real line to be constructed. The first construction was done by the a British mathematician Henry J. S. Smith in 1875, and Vito Volterra, still a graduate student in Italy, also showed how to construct such a set  in 1881. Due to Cantor's prestige, the Cantor ternary set was  (and still is) the typical example of a perfect nowhere-dense set. Following D. Bresoud \cite{bres} we will refer as the Smith-Volterra-Cantor sets or $SVC$ sets to the family of examples of perfect, nowhere-dense sets exemplified by the work of Smith, Voterra and Cantor. In the present paper we will discuss construction of several perfect nowhere-dense set in the real line.  Most of them  variations in one way or the other of the construction of the  Cantor  ternary set. Thus, it is important to review in detail the construction of the  Cantor  ternary set $C$. \\

\subsection{Cantor Ternary Set.}
$C$ is obtained from the closed interval $[0,1]$ by a sequence of deletions of open intervals known as ''middle thirds".
We begin with the interval $[0,1]$, let us call it $C_0$, and remove the middle third, leaving us with leaving us with the union of two closed intervals of length $1/3$
$$C_1= \left[0,\frac{1}{3}\right]  \cup \left[\frac{2}{3}, 1\right] .$$ Now we remove the middle third from each of these intervals, leaving us with the union of four closed intervals of length $1/9$
\begin{equation*}
C_2=\left[0,\frac{1}{9}\right] \cup \left[\frac{2}{9},\frac{1}{3}\right] \cup\left[\frac{2}{3},\frac{7}{9}\right] \cup\left[\frac{8}{9},1\right] .
\end{equation*}
Then we remove the middle third of each of these intervals leaving us with eight intervals of length $1/27$. We continue this process inductively , then for each $n=1,2, \cdots $ we get a set $C_n$  which is the union of $2^n$ closed intervals of length $1/3^n$. This iterative construction is illustrated in the following figure, for the first four steps:
\begin{center}
\includegraphics[width=3in]{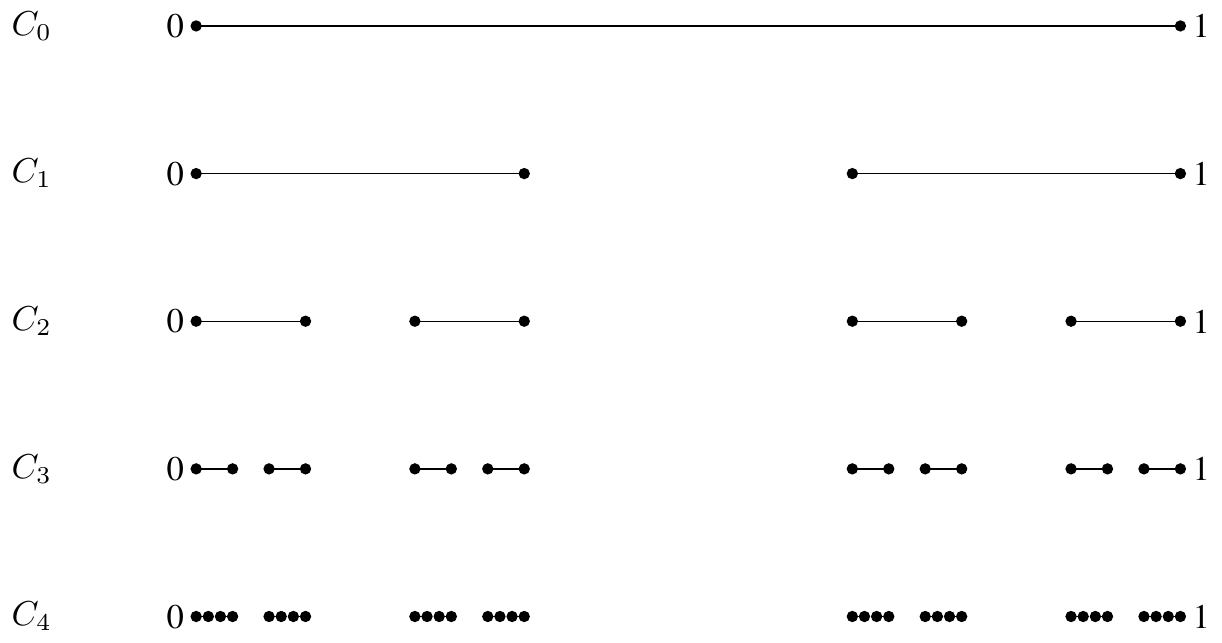}
\end{center}
Finally, we define the {\em Cantor ternary set} $C$ as the intersection
\begin{equation}
C= \bigcap_{n=0}^\infty C_n.
\end{equation}

Clearly $C \neq \emptyset $, since  trivially $0,1 \in C$. Moreover  $C$ is a a closed set, being the countable intersection of close sets, and  trivially bounded, since it is a subset of $[0,1]$. Therefore $C$ is a compact set. Moreover, observe by the construction that if $y$ is the end point of some closed subinterval of a given  $C_n$ then it is also the end point of some of the subintervals of $C_{n+1}$. Because at each stage, endpoints are never removed, it follows that $y \in C_n$ for all $n$. Thus $C$ contains all the end points of all the intervals that make up each of the sets $C_n$ (or alternatively,  the endpoints to the intervals removed)  all of which are rational ternary numbers in $[0,1]$, i.e. numbers of the form $k/3^n$. But $C$ contains much more than that; actually it is an uncountable set since it is a {\em perfect set}. To prove that simply observe that every point of $C$ is approachable arbitrarily closely by the endpoints of the intervals removed (thus for any $x\in C$ and for each $n \in \mathbb{N}$ there is an endpoint, let us call it $y_n \in C_n$, such that $|x-y_n| < 1/3^n$).

$C$ is a {\em nowhere-dense} set, that is, there are no intervals included in $C$. The easiest way to prove that is using an alternative characterization of $C$.
Consider the ternary representation for  $x \in [0,1],$ \footnote{Observe, for the ternary rational  numbers $k/3^n$ there are two possible ternary expansions, since
$$ \frac{k}{3^n} = \frac{k-1}{3^n} + \frac{1}{3^n}  =  \frac{k-1}{3^n} + \sum_{k=n+1}^\infty \frac{2}{3^k}.$$ 
Similarly, for the dyadic rational numbers $k/2^n$ there are two possible dyadic expansions as
$$ \frac{k}{2^n} = \frac{k-1}{2^n} + \frac{1}{2^n}  =  \frac{k-1}{2^n} + \sum_{k=n+1}^\infty \frac{1}{2^k}.$$ 
Thus for the uniqueness of the dyadic and the ternary representations we will take  the infinite expansions representations for the dyadic and ternary rational numbers.}
 
\begin{equation}\label{ternaryexp}
 x = \sum_{k=1}^\infty \frac{\varepsilon_k(x)}{3^k}, \quad \varepsilon_k(x)=0, 1, 2 \quad \mbox{for all} \,k = 1, 2, \cdots 
\end{equation}

Observe that removing the elements where at least one of the \(\varepsilon_{k}\) is equal to one is the same as removing the middle third in the iterative construction, thus 
the Cantor ternary set is the set of numbers in $[0,1]$ that can be written in base 3 without using the digit 1, i.e.
\begin{equation}\label{ternaryChar}
 C=\left\{x\in [0,1] : x = \sum_{k=1}^\infty \frac{\varepsilon_k(x)}{3^k}, \quad \varepsilon_k(x)=0, 2 \quad \mbox{for all} \; k = 1, 2, \cdots
 \right\}.
\end{equation}

Taking two arbitrary points in $C$ we can always find a number between them that requires the digit 1 in its ternary representation, and therefore there are no intervals included in $C$. Surprisingly,
it can be proved that, except for $0,1$, every point in $[0,1]$ can be obtained as a midpoint of two point in $C$, see \cite{rand}. 

On the other hand, observe that for each $n \in {\mathbb N}$ the set 
\begin{equation}\label{ternaryendpoints}
F_n = \left\{t \in [0,1]: t= \sum_{k=1}^n \frac{\varepsilon_k}{3^k},  \quad \mbox{with} \; \varepsilon_k=0, 2 \quad \mbox{for all} \; k = 1, 2, \cdots, n  \right\}
\end{equation}
are precisely the left endpoints of the subintervals of $C_{n}$, and therefore $C_{n}$ can be represented as,
\begin{equation}\label{CharCn}
C_{n} = \bigcup_{t \in F_n} \left[t, t+\frac{1}{3^n}\right].
\end{equation}

Also using this characterization  of  $C$ we can get a direct proof that it is uncountable. Define 
the mapping $f: C \rightarrow [0,1]$ for $  x=\sum_{k=1}^\infty \frac{\varepsilon_k(x)}{3^k} \in C$, as 
\begin{equation}\label{cantorfun}
f(x) =    \sum_{k=1}^\infty \frac{\varepsilon_k(x)/2}{2^k}= \frac{1}{2}    \sum_{k=1}^\infty \frac{\varepsilon_k(x)}{2^k}.
\end{equation}

It is clear that $f$ is one-to-one correspondence from $C$ to $[0,1]$ (observe that as $\varepsilon_k=0, 2$ then $\varepsilon_k/2=0, 1$).

Finally, $C$ has measure zero, since
$$ m(C) = m([0,1]) - m(C^c) = 1 - \sum_{n=1}^\infty \frac{2^{n-1}}{3^n} = 1 -\frac{1}{3}  \sum_{n=0}^\infty (\frac{2}{3})^{n} = 1 -  \frac{1/3}{1- 2/3}=1-1=0.$$

Observe that $C$ can then be obtained by removing a fixed proportion (one third) of each subinterval in each of the iterative steps, by removing the length $1/3^k$ from each subinterval in the $k^\text{th}-$step, or removing  the digit one of the ternary expansion. Of course these three constructions are equivalent. Nevertheless, the first construction, removing a fixed proportion of each subinterval in each of the iterative steps, give us another important property of the Cantor set, its self-similarity across scales, this is illustrated in the following figure:
\begin{center}
\includegraphics[width=3in]{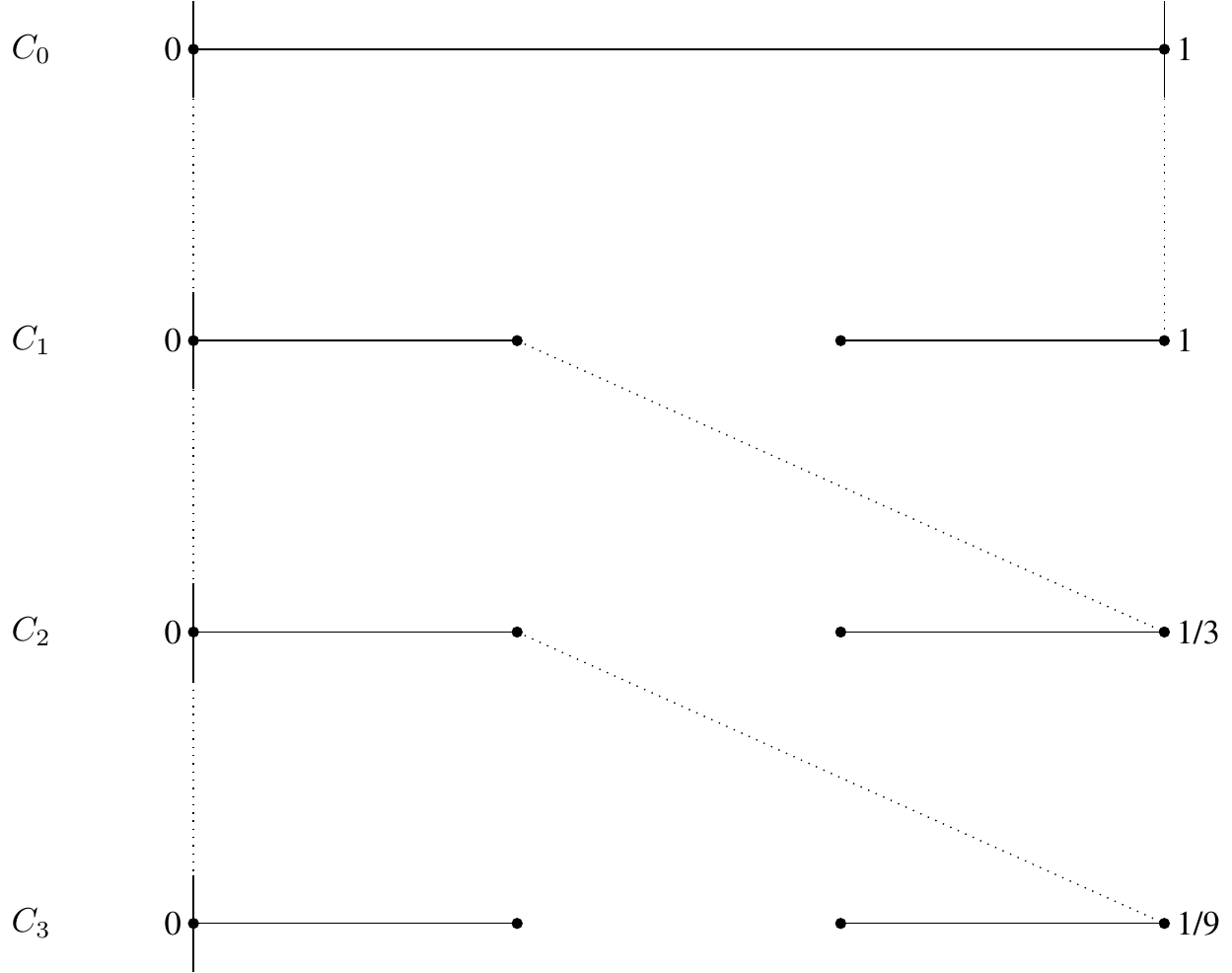}
\end{center}

\subsection{Cantor Function} We can extend the function defined in (\ref{cantorfun}) to  a function defined on $[0,1]$.  For $x \in [0,1]$  written in is ternary expansion (\ref{ternaryexp}) then if 
$$N(x) = \min\{k: \varepsilon_{k}(x)=1\},$$
then 
\begin{equation}
G(x)=\sum_{k=1}^{\infty} \frac{\gamma_{k}(x)}{2^{k}}
\end{equation}
 where
 $$\gamma_{k}(x) = \left\{ \begin{array}{rcl} \frac{\varepsilon_{k}(x)}{2} & \mbox{for} & k<N(x) \\ 0 & \mbox{for} & k=N(x). \\ 1 & \mbox{for} & k>N(x)  \end{array}\right.$$

Observe that this implies that $G$ can be written as
$$ G(x) =\frac{1}{2}\sum_{k=1}^{N(x)-1} \frac{\varepsilon_{k}(x)}{2^{k}}+\sum_{k=N(x)+1}^{\infty} \frac{1}{2^{k}} = \frac{1}{2}\sum_{k=1}^{N(x)-1} \frac{\varepsilon_{k}(x)}{2^{k}}+\frac{1}{2^{N(x)}} $$

This function is called the {\em Cantor function} and it was also defined by G. Cantor in 1883 to be used used as a counterexample to some of Harnack's claim on an extension of the Fundamental Theorem of Calculus. Also H. Lebesgue considered it in his work on integration in 1904, and for that reason it is sometimes refers as {\em  Lebesgue  singular function}. Finally, it is also known with  a more descriptive name as the {\em Devil's staircase}.  

Moreover, if  $t = \sum_{k=1}^{n-1} \frac{\varepsilon_{k}(x)}{3^{k}} \in F_{n-1} $ is the left endpoint of the subinterval of $C_{n-1}$ containing $x$, then
$$G(x)= G(t) +\frac{1}{2^{N(x)}}=G( t+\frac{2}{3^n}).$$

Observe also that if $t = \sum_{k=1}^{n-1} \frac{\varepsilon_{k}}{3^{k}} \in F_{n-1},$
\begin{equation*}
G(t +\frac{1}{3^{n}}) = G(t+ \sum_{k=n+1}^\infty \frac{2}{3^k}) = \sum_{k=1}^{n-1} \frac{\varepsilon_{k}/2}{2^{k}} +  \sum_{k=n+1}^\infty \frac{1}{2^k} =  \sum_{k=1}^{n-1} \frac{\varepsilon_{k}/2}{2^{k}} + \frac{1}{2^{n}} = G(t + \frac{2}{3^{n}}).
\end{equation*}
Thus $G$ is constant on the open subintervals $(t +\frac{1}{3^{n}}, t + \frac{2}{3^{n}})$ which are those that are removed on the $n^\text{th}$-step. Also, from this construction it is easy to see that $f$ is monotone non-decreasing.\\

There is also an iterative construction for the Cantor function, that allows us to get the properties of $G$ easily.
$$\psi_{0}(x)= x, \; \mbox{ for } 0{\leq}x{\leq}1,$$
and for $n \geq 1$
$$\psi_{n}(x) = \left\{ \begin{array}{rcl} \frac{1}{2} \psi_{n-1}(3x) & \mbox{for} & 0{\leq}x{\leq}\frac{1}{3} \\
    							\frac{1}{2} & \mbox{for} & \frac{1}{3}{\leq}x{\leq}\frac{2}{3} \\
								\frac{1}{2}\psi_{n-1}(3x-2)+\frac{1}{2} & \mbox{for} & \frac{2}{3}{\leq}x{\leq}1
		\end{array}\right.$$
Observe that, similar to removing middle thirds in the construction of $C$, the Cantor ternary set, the middle thirds are made constant in the iterative construction of the Cantor function. Also $f_n$ is a monotone non-decreasing continuous function for all $n \geq 1$. It can be proved that, if $m < n$

$$|\psi_n(x) - \psi_m(x) | \leq \frac{1}{2^m}, \; \mbox{for all} \; x \in [0,1]. $$ 
By the Cauchy criterion for uniform convergence, we conclude that $\{\psi_n\}$ converges uniformly. Then $G$ is a monotone non-decreasing continuous function on $[0,1]$. Finally observe that $G$ is constant in all the intervals that are removed in the construction of $C$, $G(0) =0$ and $G(1) =1$ so increases only in the Cantor set, thus $G$ is a {\em singular function } i.e. $G' = 0$ a.e. 

It can be proved that $G$ is the unique uniformly real-valued function on $[0,1]$ such that  
$$G(x) = \left\{ \begin{array}{rcl} \frac{1}{2} G(3x) & \mbox{for} & 0{\leq}x{\leq}\frac{1}{3} \\
								\frac{1}{2} & \mbox{for} & \frac{1}{3}{\leq}x{\leq}\frac{2}{3} \\
								\frac{1}{2}G(3x-2)+\frac{1}{2} & \mbox{for} & \frac{2}{3}{\leq}x{\leq}1
		\end{array}\right.$$
For the proof of this result and several other interesting results about the Cantor function, see an extensive study on it at  \cite{DovMartRyazVour}. Also, the Cantor function can be characterized as the only monotone increasing function on $[0,1]$ such that $G(0)=0, \, G(x/3)=G(x)/2$ and $G(1-x)= 1 G(x),$ see \cite{chalice}.

Illustration of some iteration of the construction of  the Cantor function $G$ is given in the following figure,
\begin{center}
\includegraphics[width=3in]{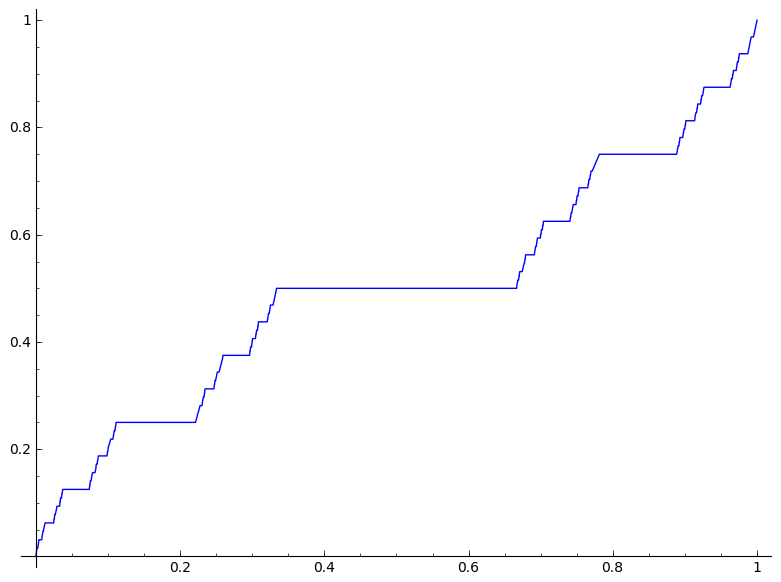}
\end{center}

\section{Variations and generalizations of the Cantor ternary set.}

In this section we are going to consider several variations and generalizations of the Cantor set, of course all of them are bounded, perfect, nowhere-dense sets. But first, it is interesting to observe that all the nowhere-dense perfect sets are constructed in a similar form as the Cantor sets as the following result shows, see \cite{bres}. For completeness the proof of it is included.

\begin{theorem}
If $S$ is a bounded, perfect, nowhere-dense, non-empty set, a =$ \inf S$, b =$\sup S$, then there is a countable collection of open intervals contained in $[a, b]$ such that $S$ is the  derived set (i.e. the set of all accumulation points) of the set of the endpoints of these intervals.
\end{theorem}

\begin{proof}
Since $S$ is perfect, it is closed thus $a, b \in S$.  Then $S^c \cap [a,b]$, its complement in $[a,b]$, is relative open, and therefore it is a countable union of open intervals. Since each point of $S$ is an accumulation point of $S$, $a$ can not be a left endpoint of one of these open intervals nor can $b$ be a right endpoint of one of these open intervals. If $S^c \cap [a,b]$ consisted of only finitely many intervals, then $S$ would contain a closed interval of positive length. Since $S$ is nowhere-dense, the number of intervals in $S^c \cap [a,b]$ must be countably infinite. 

Now we will show that  $S$ is the set of all accumulation points of the set of the endpoints of the disjoint open intervals whose union is $S^c \cap [a,b]$. Let $\{I_1, I_2, \cdots \}$ be an enumeration of these disjoint open intervals. For the intervals $I_j$, let $L(I_j)$ and $R(I_j)$ be the left and right endpoints, respectively. The set of endpoints is contained in $S$, and, since $S$ is perfect, its derived set it is also in $S$. Since $S$ is a nowhere-dense set, we know that is closure $\overline{S}=S$ does not contain any open interval. Thus if $s \in S$ every neighborhood of $s$ has a nonempty intersection with at least one of the intervals $I_j$ and therefore an endpoint of this interval is in the neighborhood of $s$. It follows that $s$ is an accumulation point for the set of endpoints, and, therefore, $S$ is the derived set of endpoints.
\end{proof}

Let us start discussing generalizations of the Cantor set.  Observe that due to the previous theorem since all these constructions are obtained by removing a countable number of open intervals from the interval $[0,1]$ they are bounded, perfect, nowhere-dense, non-empty sets.\\

In  \cite{bres}, D. Bresoud considers the following family of examples of perfect, nowhere-dense sets.

\begin{defi} The set $S$ is said to be a {\em Smith-Volterra-Cantor} set or simply an $SVC$ set, if it is  constructed by starting with a closed interval  and removing an open subinterval. One then removes an open interval from each of the remaining subintervals and continue through an infinite sequences of such removals, choosing the subintervals that are removed such that every open subinterval of the original set overlaps with at least one of the subintervals that are removed. The $SVC$ set  $S$ is the intersection of the countably infinite collection of sets that remain after each iteration.
\end{defi}

Observe that, by construction, $S$ is a perfect and nowhere-dense set. \\

A particular family of $SVC$ sets is defined, for each $n\geq 3$, as the $SCV(n)$ set. They are constructed by removing, in the $k^\text{th}$ step of the iteration, an open interval of length $1/n^k$ from the center of each of the remaining closed intervals. Observe that $SVC(3)$ is precisely the Cantor ternary set.

 Another important case is $SVC(4)$, which was considered by Volterra as he constructed his famous counter-example of a function with bounded derivative that exists everywhere but the derivative is not Riemann integrable in any closed and bounded interval, see \cite{bres} Chapter 4. 
 
 The $SCV(n)$ sets are also perfect sets and the measure of each is
$$ m([0,1])-m(SVC(n)^{c}) = 1- \frac{1}{n} - \frac{2}{n^2} - \frac{4}{n^3} - \cdots = 1- \frac{1}{n} \frac{1}{1-2/n} = \frac{n-3}{n-2}.$$

Thus the $SCV(n)$ sets have positive measure, and moreover they can have measure very close to one,  but nevertheless they are nowhere-dense. For this reason they are called {\em fat Cantor sets}.\\

\subsection{ $\lambda-$fat Cantor sets.}

Another type of fat Cantor set can be found in the literature, see \cite{GelOlm}, \cite{darst} and  \cite{BoeDarErd}. Take $ 0< \lambda \leq 1$ and repeat the construction of $C$ with the following modification. At the $k^\text{th}$ step, instead of taking out an open interval of length $1/3^k$ from the center of each of the $2^{k-1}$ intervals of equal length, take out an open interval of length $\lambda/3^k$ to obtain $C_{\lambda,k}$ ($\lambda/2^{2k-1}$ in \cite{GelOlm}). We then obtain $C_{\lambda}$ as usual as
\begin{equation}
C_{\lambda}= \bigcap_{n=0}^\infty C_{\lambda,n}.
\end{equation}
Observe that the sum of the length  open intervals taken out in the first $k$ steps is
$$ \frac{\lambda}{3} + 2\frac{\lambda}{3^2} +\cdots + 2^{k-1}\frac{\lambda}{3^k} = \lambda [1- (\frac{2}{3})^k],$$
thus the measure of $C_\lambda$ is $1- \lambda.$\\

A variation of the previous construction is seen in the following family of sets.\\

\subsection{Middle$-\beta$ Cantor sets.}

Let $0<\beta <1$, then begin with the unit interval $[0,1]$.  In the first iteration, remove an open subinterval of length $\beta$ of $[0,1]$ from the middle of the interval so that,
 $$C^{\beta}_{1}=\left[ 0,\frac{1}{2}(1-\beta) \right] \cup \left[ \frac{1}{2}(1+\beta),1 \right]$$
In the second iteration, remove a $\beta$ proportional length  of each of the two disjoint intervals from the middle of each of the intervals in $C^{\beta}_{1}$. Then
\begin{eqnarray*}
C^{\beta}_{2}&=&\left[ 0,\frac{1}{4}(1-\beta)^{2} \right] \cup \left[ \frac{1}{4}(1-\beta)(1+\beta),\frac{1}{2}(1-\beta) \right]\\
&& \quad \quad \cup\left[ \frac{1}{2}(1+\beta),\frac{1}{2}\left( (1+\beta)+\frac{1}{2}(1-\beta)^{2} \right) \right] \cup \left[ \frac{1}{2}(1+\beta)\left( 1+\frac{1}{2}(1-\beta) \right), 1 \right].
\end{eqnarray*}

Continue iterating by removing an open subinterval of length $\beta-$proportional of each of the disjoint intervals from the middle of each of the intervals.  Thus we are removing $\beta-$proportional subintervals each time. Then, the $\beta-$Cantor set,
$$C^{\beta}=\bigcap_{k=1}^{\infty}C^{ \beta}_{k}.$$
The next figure illustrate the first three steps of this construction,
\begin{center}
\includegraphics[width=3in]{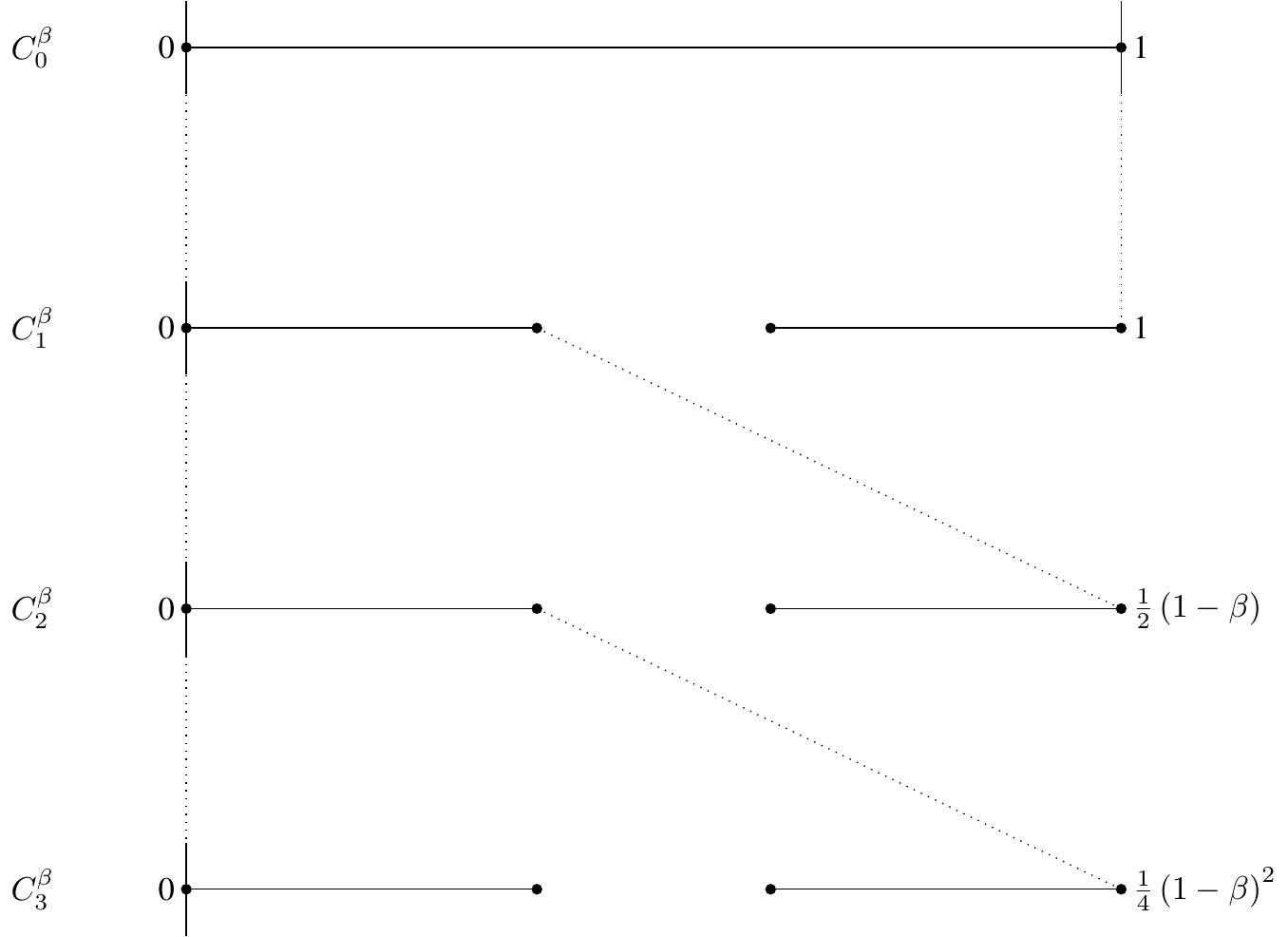}
\end{center}
Observe from the figure that the ratio between the closed intervals and the removed open subintervals remains constant, i. e. we have in this case self-similarity.

The $\beta$-Cantor  sets have measure zero. We can prove this using the continuity from above of the Lebesgue measure, as $m({C^{\beta}_{k}}) = (1-\beta)^{k} $ and therefore
$$m\left({C^{\beta}} \right) = \lim_{k \to \infty} m\left( {C^{\beta}_{k}} \right) = \lim_{k \to \infty} (1-\beta)^{k}  = 0,$$
or, as usual, computing the measure of the set from the measure of its complement in $[0,1]$,
$$ m\left( {C^{\beta}} \right)= 1- \beta - 2 (\frac{1}{2} (1-\beta) \beta) - 4(\frac{1}{4} (1-\beta)^2\beta) - \cdots = 1 - \beta \frac{1}{1-(1-\beta)} =1-1 =0.$$ 

\subsection{Middle$-\{\beta_{i}\}$ Cantor sets.}

The $C^{\beta}$ sets can be generalized by allowing the proportion removed from the intervals to change at each iteration in the construction.  In this construction, a sequence $\{\beta_{i}\}$ is chosen with each $0<\beta_{i}<1$.  Then starting again with $$C_{0}^{\{\beta_{i}\}} = [0,1],$$ in the first iteration we remove $\beta_{1}$ from the middle of the interval. Then $$C_{1}^{\{\beta_{i}\}} = \left[ 0,\frac{1}{2}(1-\beta_{1}) \right] \cup \left[ \frac{1}{2}(1+\beta_{1}),1 \right].$$
In the next iteration, $\beta_{2}$ is removed from the middle of the two disjoint intervals. Then,
\begin{eqnarray*}
C^{\{\beta_{i}\}}_{2}&=&\left[ 0,\frac{1}{4}(1-\beta_{1})(1-\beta_{2}) \right] \cup \left[ \frac{1}{4}(1-\beta_{1})(1+\beta_{2}),\frac{1}{2}(1-\beta_{1}) \right]\\
&& \quad \quad \cup\left[ \frac{1}{2}(1+\beta_{1}),\frac{1}{2}\left( (1+\beta_{1})+\frac{1}{2}(1-\beta_{1})(1-\beta_{2} \right) \right]\\
&& \quad \quad \cup \left[ \frac{1}{2}(1+\beta_{1})\left( 1+\frac{1}{2}(1-\beta_{2}) \right), 1 \right].
\end{eqnarray*}
Iterating the process as before, we get
$$C^{\{\beta_{i}\}}=\bigcap_{j=1}^{\infty} C^{\{\beta_{i}\}}_{j}.$$

In order to determine the measure of $C^{\{\beta_{i}\}}$, the measure of each iteration in the construction should be noted.

\begin{eqnarray*}
m(C^{\{\beta_{i}\}}_{0})&=&1\\
m(C^{\{\beta_{i}\}}_{1})&=&(1-\beta_{1})\\
m(C^{\{\beta_{i}\}}_{2})&=&(1-\beta_{1})(1-\beta_{2})\\
&\vdots&\\
m(C^{\{\beta_{i}\}}_{n})&=&(1-\beta_{1})(1-\beta_{2})(1-\beta_{3})\cdots(1-\beta_{n}),\\
\end{eqnarray*}
therefore
\begin{equation}
m(C^{\{\beta_{i}\}})=\prod_{j=1}^{\infty}(1-\beta_{j})
\end{equation}

Depending on how the $\{\beta_{i}\}$ are chosen, $C^{\{\beta_{i}\}}$ can either be a fat Cantor set or have measure zero.\\

Now we introduce a class of Cantor-like sets that includes all the $SVC(n)$ sets, all the $C_\lambda$ sets and  all the $C^{\beta}$ sets.  We will call them {\em generalized} $SVC$ sets.\\

\subsection{ {\em Generalized} $SVC(\alpha,\beta)$ sets.}

Consider $0< \alpha, <1$ and  $0 < \beta<1$ such that $ 0 < \beta \leq 1- 2\alpha$. We define a {\em generalized} $SVC(\alpha,\beta)$  set, $C^{(\alpha, \beta)}$, as follows.
Begin with the unit interval $[0,1]$.  In the first iteration, we remove an open interval of length $\beta$ centered at $1/2$ from $[0,1]$, obtaining
 $$C^{(\alpha, \beta)}_{1}=\left[ 0,\frac{1}{2}(1-\beta) \right] \cup \left[ \frac{1}{2}(1+\beta),1 \right]$$
In the second iteration, we remove an open interval of length $\alpha\beta$ of each of the two disjoint intervals from the middle of each of the intervals in $C^{(\alpha, \beta)}_{1}$, 
\begin{eqnarray*}
C^{(\alpha, \beta)}_{2}&=&\left[ 0,\frac{1}{4}(1-\beta)- \frac{1}{2} \alpha \beta \right] \cup \left[ \frac{1}{4}(1-\beta)+\frac{1}{2}\alpha \beta ,\frac{1}{2}(1-\beta) \right]\\
&& \quad \cup\left[ \frac{1}{2}(1+\beta),\frac{1}{2}\left( (1+\beta)+\frac{1}{2}(1-\beta) -\alpha \beta\right) \right] \cup \left[\frac{1}{2}\left( (1+\beta)+\frac{1}{2}(1-\beta) +\alpha \beta\right), 1 \right]
\end{eqnarray*} 

We iterate this procedure by removing in the $k^\text{th}$- step an open subinterval of length $\alpha^{k-1}\beta$ from the middle of each of the disjoint intervals.  Thus we are removing $2^{k-1}$ subintervals of length $\alpha^{k-1}\beta$ each time. Then, the {\em generalize} $SVC(\alpha,\beta)$ set, $C^{(\alpha, \beta)}$, is obtained as
$$C^{(\alpha, \beta)}=\bigcap_{k=1}^{\infty}C^{(\alpha, \beta)}_{k}.$$
The next figure illustrates the first three steps of this construction. Observe that in this case there is not proportionality in the construction,
\begin{center}
\includegraphics[width=3in]{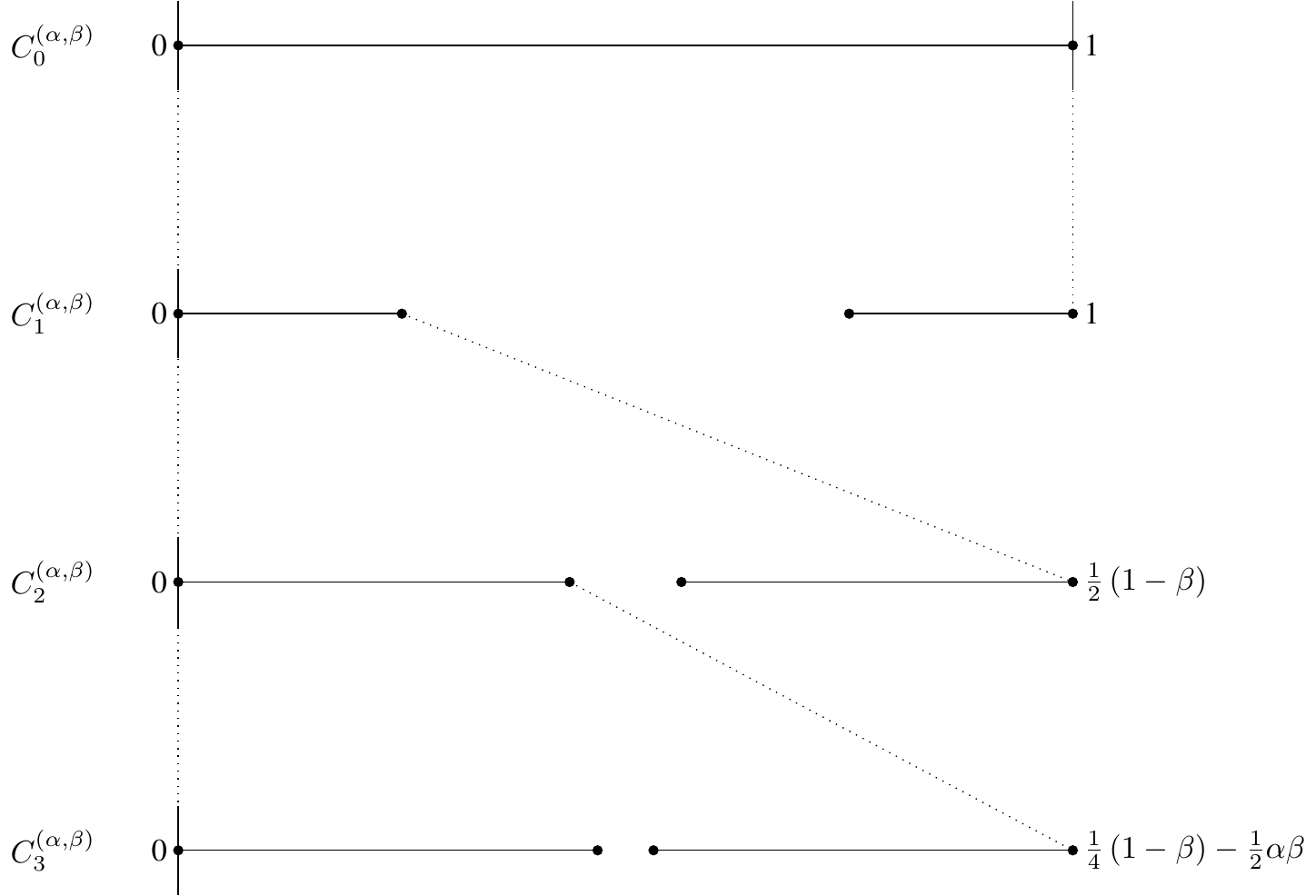}
\end{center}

The measure of the $C^{(\alpha, \beta)}$-Cantor  sets is then the following
$$ m\left( C^{(\alpha, \beta)} \right)= 1- \beta - 2 \alpha \beta - 4 \alpha^2\beta - \cdots = 1 - \frac{ \beta}{1- 2\alpha}.$$ 
From here additional conditions on $\alpha, \beta$ arise $1- 2 \alpha >0$ thus $0 < \alpha < 1/2$ and 
$ 0 < \beta \leq 1- 2\alpha.$ Therefore $C^{(\alpha, \beta)}$-Cantor  sets are fat sets unless $ \beta =1- 2\alpha.$\\

Observe that the Cantor  ternary set $C$ is obtained from this case taking $\alpha= \beta= 1/3$; the $C_\lambda$ Cantor sets are obtained taking $\alpha=1/3, \beta = \lambda/3$; the $C^\beta$ Cantor sets are obtained simply taking $\alpha = (1-\beta)/2, \beta=\beta$ and the $SVC(n)$ sets taking $\alpha= \beta = 1/n$.\\

\subsection{ {\em Cantor-like} sets determined by a sequence $\{\gamma_j\}$.}

In \cite{GelOlm} there is a very general construction. Let $0<\gamma<1$ and a sequence of positive numbers $\{\gamma_j\}$ such that $\sum_{j=0}^\infty 2^j \gamma_j = \gamma.$ Delete from $[0,1]$ an open interval $I_0$ centered at $1/2$ and of length $\gamma_0$. Then from $[0,1] -I_0$ delete two open intervals $I_1^1, I^2_1,$ each one centered in one of the two disjoint, closed intervals whose union is 
$[0,1] -I_0$ and each of length $\gamma_1$, continuing deletions as in preceding steps. At the $j^\text{th}-$step of deletion, $2^j$ open subintervals, $I^1_j, I^2_j, \cdots I_j^{2^j}$, properly centered in the closed intervals constituting the residue at the $(j-1)^\text{th}$-step, and each of length $\gamma_j$ are deleted.

Observe that the Cantor ternary set $C$ is obtained from this construction by taking $\gamma_j = 1/3^{j+1}$. The $C_\lambda$- Cantor sets are obtained by taking $\gamma_j = \lambda / 3^{j+1}$, the $SVC(n)$ sets by taking $\gamma_j = 1/n^{j+1},$ and the generalized $SVC$ sets  by taking $\gamma_j = (2 \alpha)^j \beta.$\\

An analogous construction can also be  found in \cite{kardos}.\\

There is a generalization of this construction due to Besicovitch and Taylor, see \cite{BesicoTaylor}, given in a non-constructive way as follows. $E$ is a closed set contained in $[0,1]$. Let $I =(0,1)$, then $I-E$ is an open set and therefore it is a countable union of open intervals. If the length of these intervals may be arranged as a non-increasing sequence $\psi= \{a_n\}$ of positive numbers, such that 
\begin{equation}\label{measure0cond}
\sum_{n=1}^\infty a_n =1,
\end{equation}
then $E$ has zero Lebesgue measure. Observe that given any non-increasing sequence $\psi= \{a_n\}$ satisfying (\ref{measure0cond}), there are many possible sets $E$ such that the complementary intervals $I-E$ have lengths given by the elements of $\psi$, but not all of them are necessarily perfect sets.The Cantor set depends of the order of the sequence and the Hausdorff dimension of the resulting set also depend on the order.  \\

\subsection{ $k$-adic-type Cantor-like sets.}

The alternative characterization of the Cantor ternary set (\ref{ternaryChar}) allows us to extend another generalization of it. Let $k \in \mathbb{N}$ with $k>2$, $p \in \mathbb{N}, \; 1<p<k$, and start again with the unit interval $[0,1]$.  Partition the interval into $k$  intervals of equal length and remove $p$ open intervals of the partition making sure to leave the first and last intervals\footnote{ if not, the set obtained would not be perfect since the construction would produce isolated points.} obtaining $C^k_{1}$. In the second iteration, partition each of the remain $k-p$ intervals in $C^k_{1}$ into $k$ partitions of equal length and remove the same $p$ open intervals of the partition corresponding to the intervals removed in the first iteration, obtaining $C^k_{2}$.  Continue partitioning each subinterval in $k$ subinterval and removing $p$ of them in the same pattern.  Then
$$C^k=\bigcap_{i=i}^{\infty} C^k_{i}.$$
Observe that in this construction in each step $p \geq 1$ subintervals are removed from each closed subinterval in  $C^k_{i}$.\\
   
Alternatively, $k$-adic Cantor  set $C^k$,  can be defined directly from the $k$-adic expansion of the real numbers in $[0,1]$.

\begin{eqnarray*}
&& C^k = \{ x \in [0,1] :  x=\sum_{i=1}^{\infty}\frac{\epsilon_{i}}{k^{i}}\,  \text{where }\epsilon_{i} \in \{\eta_0, \eta_1, \cdots,\eta_p\} \subseteq \{0,1,\cdots,k-1\}, \\
&& \hspace{8cm} \text{ and }\, \eta_1=0, \eta_p=k-1\} 
\end{eqnarray*}

The $k$-adic-type sets have measure zero, since for $n\geq 1,$
$m\left( C^k_{i} \right) =\left(\frac{k-p}{k}\right)^{ni}$ and therefore,

$$m\left( {C^k} \right)  = \lim_{i \to \infty} m\left( {C^k_{i}} \right) = 0.$$

\subsection{ {\em Rescaling Cantor sets.} }

A generalization of the $k$-adic-type Cantor-like sets is give in the following construction, see \cite{Tsang}.  Let $\mathcal{C}_0$ be the closed interval $[0,1]$ and $p \in {\mathbb N}$. Form the  set $\mathcal{C}_1$ by deleting $p-1$ open intervals from $\mathcal{C}_0$ so that the $p$ closed intervals, each of length $\varepsilon_i>0 \; (i =1, \cdots, p)$ of the mother interval remain. Form $\mathcal{C}_2$ by repeating this process with each of the $k$ intervals in $\mathcal{C}_1$.  Continue this process inductively. The set 
$$\mathcal{C} = \bigcap_n \mathcal{C}_n$$
 is called {\em rescaling Cantor set}, since each ``daughter" resembles its mother. The $k$-adic type Cantor set is obtained when $\varepsilon_i= \frac{1}{k}$, for all $i=1,\cdots,p$. If the '$\varepsilon_i$'s are different the set $\mathcal{C} $ is called a {\em non-uniform} Cantor set.
 
\vspace{5mm}

\subsection{ {\em Cantor-like} sets determined by two sequences $\{k_n\}$, $\{c_n\}$. }

Finally, there is a very general procedure to construct Cantor-like sets, see \cite{GanVijaVen}. 
Let $\{k_n\}$ be a sequence of positive integers such that $k_n \geq 2,$ for every $n \in {\mathbb N}$. Let $\{c_n\}$ be another sequence  of positive integers such that $0<k_n c_n < c_{n-1},$ for every $n \geq 1.$ For each $n \in {\mathbb N}$ define 
$$r_n = \frac{c_{n-1}-c_n}{k_n-1}.$$
Let $F_0 =\{0\},$ and for $n \geq 1$
$$ F_n = \left\{ t= \sum_{j=1}^n \varepsilon_j r_j:  \quad \varepsilon_j=0, 1,2, \cdots, k_n-1, \quad \mbox{for all} \; j = 1, 2, \cdots, n  \right\}.$$
Define 
$$ E^{k,c}_n = \bigcup_{t \in F_n} [t, t+ c_n].$$
Then the set 
$$ E^{k,c} = \bigcup_{n=1}^\infty E^{k,c}_n,$$
is called the {\em Cantor-like} set determined by the sequences $\{k_n\}$ and $\{c_n\}.$ When $k_n =2$ for all  $n \in {\mathbb N}$ these set are called {\em symmetric Cantor} sets determined by the sequence $\{c_n\}$, see \cite{LeePark}. \\

As we already said this construction is very general, for instance:

\begin{itemize}
\item The Cantor ternary set corresponds to the case, $k_n =2,$ $c_0=1$ and $c_n = 1/3^n$, for all  $n \geq 1.$ In this case $r_0=1$ and $r_n =  2/3^n$ for all  $n \geq 1.$ \\

\item The $\lambda$ - fat Cantor sets correspond to the case $k_n =2,$ $c_0=1$ and 
$$c_n = \frac{1}{2^n} (1- (1-(\frac{2}{3})^n) \lambda),$$
 for all  $n \geq 1.$ In this case $r_0=1$ and 
 $$r_n =  \frac{1}{2^n} (1- \lambda +2(\frac{2}{3})^n \lambda),$$ 
 for all  $n \geq 1.$

\item The middle$-\beta$ Cantor sets correspond to the case $k_n =2,$ $c_0=1$ and 
$$c_n = \frac{1}{2^n} (1- \beta)^n,$$ 
for all  $n \geq 1.$ In this case $r_0=1$ and 
$$r_n =  \frac{1}{2^n} (1- \beta)^{n-1} (1+\beta),$$ 
for all  $n \geq 1.$

\item The generalized middle$-\{\beta_{i}\}$ Cantor sets correspond to the case $k_n =2,$ $c_0=1$ and 
$$c_n = \frac{1}{2^n} \prod_{j=1}^n(1- \beta_j),$$ 
for all  $n \geq 1.$ In this case $r_0=1$ and 
$$r_n =  \frac{1}{2^n} \prod_{j=1}^{n-1}(1- \beta_j) (1+\beta_n),$$ 
for all  $n \geq 1.$

\item The generalized $SVC(\alpha, \beta)$ sets correspond to the case $k_n =2,$ $c_0=1$ and 
$$c_n = \frac{1}{2^n} (1-\beta) -\frac{\alpha \beta}{2^{n-1}}\left[\frac{1-(2\alpha)^{n-1}}{1-2\alpha}\right] ,$$ for all  $n \geq 1.$ In this case $r_0=1$ and 
$$r_n =  \frac{1}{2^{n}} (1- \beta)- \frac{\alpha \beta}{1-2\alpha}\left[\frac{1}{2^{n-1}}-(1-\alpha)\alpha^{n-2}\right], $$ for all  $n \geq 1.$\\

\end{itemize}

Nevertheless, the $k$-adic-type Cantor-like sets are not necessary included in this construction. There are Cantor sets even more general than these called {\em homogeneous perfect} sets, see \cite{WenWu}.\\

Finally, one of the most important results in \cite{GanVijaVen} is the computation of Hausdorff dimension of Cantor-like sets determined by two sequences, see Collorary 4, page 202 in this paper, see also \cite{kardos}. Using this result we can compute the Hausdorff dimension of several of the Cantor-like sets considered above:

\begin{itemize}
\item For the ternary Cantor set, since $k_n =2,$ $c_0=1$ and $c_n = 1/3^n$, for all  $n \geq 1,$ then as
$$a= \lim_{n\rightarrow \infty} \frac{c_n}{c_{n-1}} =   \lim_{n\rightarrow \infty} \frac{3^{n-1}}{3^{n}} = \frac{1}{3},$$
and therefore,
$$ \mbox{dim}_{\mathcal H} (C) = \frac{\log k}{-\log a} = \frac{\log 2}{-\log 1/3} =\frac{\log 2}{\log 3}.$$ 

\item For the $\lambda-$fat Cantor sets, since $k_n =2,$ $c_0=1$ and 
$$c_n = \frac{1}{2^n} (1- (1-\frac{2}{3})^n \lambda),$$
 for all  $n \geq 1,$ then as
$$a= \lim_{n\rightarrow \infty} \frac{ \frac{1}{2^n} (1- (1-(\frac{2}{3})^n) \lambda)}{ \frac{1}{2^{n-1}} (1- (1-(\frac{2}{3})^{n-1}) \lambda)} =   \lim_{n\rightarrow \infty}  \frac{ (1- (1-(\frac{2}{3})^n) \lambda)}{ 2 (1- (1-(\frac{2}{3})^{n-1}) \lambda)} = \frac{1}{2},$$
and therefore, we get the well known value,
$$ \mbox{dim}_{\mathcal H} (C_\lambda) = \frac{\log k}{-\log a} = \frac{\log 2}{-\log 1/2} =\frac{\log 2}{\log 2}=1.$$ 

\item For the middle$-\beta$ Cantor sets, since $k_n =2,$ $c_0=1$ and 
$$c_n = \frac{1}{2^n} (1- \beta)^n,$$ 
for all  $n \geq 1,$ then as
$$a= \lim_{n\rightarrow \infty} \frac{ \frac{1}{2^n} (1- \beta)^n}{ \frac{1}{2^{n-1}} (1- \beta)^{n-1}} =  \frac{1}{2}  (1- \beta),$$
and therefore,
$$ \mbox{dim}_{\mathcal H} (C^{\beta}) = \frac{\log k}{-\log a} = \frac{\log 2}{-\log (1/2(1-\beta))} =\frac{\log 2}{\log 2-\log (1-\beta)}.$$ 

\item For the generalized middle$-\{\beta_{i}\}$  Cantor sets, since $k_n =2,$ $c_0=1$ and 
$$c_n = \frac{1}{2^n} \prod_{j=1}^n(1- \beta_j),$$ 
for all  $n \geq 1,$ then as
$$a= \lim_{n\rightarrow \infty} \frac{\frac{1}{2^n} \prod_{j=1}^n(1- \beta_j)}{ \frac{1}{2^{n-1}} \prod_{j=1}^{n-1}(1- \beta_j)} =  \frac{1}{2}  \lim_{n\rightarrow \infty} (1- \beta_n),$$
and therefore,
$$ \mbox{dim}_{\mathcal H} (C^{\{\beta_{i}\}}) = \frac{\log k}{-\log a} = \frac{\log 2}{-\log (1/2 \lim_{n\rightarrow \infty} (1- \beta_n))} =\frac{\log 2}{\log 2-\log ( \lim_{n\rightarrow \infty} (1- \beta_n))}.$$ 

\item For the generalized  $SVC(\alpha, \beta)$ sets, since $k_n =2,$ $c_0=1$ and 
$$c_n = \frac{1}{2^n} (1-\beta) -\frac{\alpha \beta}{2^{n-1}}\left[\frac{1-(2\alpha)^{n-1}}{1-2\alpha}\right] ,$$
 for all   $n \geq 1,$ then as
\begin{eqnarray*}
a&=& \lim_{n\rightarrow \infty} \frac{\frac{1}{2^n} (1-\beta) -\frac{\alpha \beta}{2^{n-1}}\left[\frac{1-(2\alpha)^{n-1}}{1-2\alpha}\right] }{\frac{1}{2^{n-1}} (1-\beta) -\frac{\alpha \beta}{2^{n-2}}\left[\frac{1-(2\alpha)^{n-2}}{1-2\alpha}\right] } =   \lim_{n\rightarrow \infty} \frac{ (1-\beta) -2\alpha \beta \left[\frac{1-(2\alpha)^{n-1}}{1-2\alpha}\right] }{2 (1-\beta) -4\alpha \beta\left[\frac{1-(2\alpha)^{n-2}}{1-2\alpha}\right] } \\
&=& \frac{ (1-\beta) - \left[\frac{2\alpha \beta}{1-2\alpha}\right] }{2 (1-\beta) -\left[\frac{4\alpha \beta}{1-2\alpha}\right] } =\frac{1}{2}
\end{eqnarray*}
and therefore,
$$ \mbox{dim}_{\mathcal H} (C^{(\alpha, \beta)}) = \frac{\log k}{-\log a} = \frac{\log 2}{-\log 1/2}= \frac{\log 2}{\log 2}=1. $$ 

\end{itemize}

Observe, from the examples above, that when the Cantor-like sets have positive measure their Hausdorff dimensions are one.

\section{Variations and generalizations of the Cantor function.}

Observe that for each of the Cantor-like sets considered in the previous section we can construct iteratively  a {\em Cantor-like function}, such that for the $n^\text{th}$-step $f_n$ is a continuous, monotone, non-decreasing function, $f_n(0)=0, f_n(1) =1$ and such that $f_n$ is constant in the in the $2^{n-1}$ open subintervals removed in that step. The uniform limit of $\{f_n\}$ will be the corresponding Cantor-like function $f$. It is a monotone, non-decreasing, continuous function on $[0,1]$, $f(0)=0, f(1) =1$ and if the corresponding Cantor set has measure zero, then $f$ is singular. 

In the following figures we give illustration of some iterations of the construction of the Cantor function corresponding to the $5$-adic-type Cantor-like set,
\begin{center}
\includegraphics[width=3in]{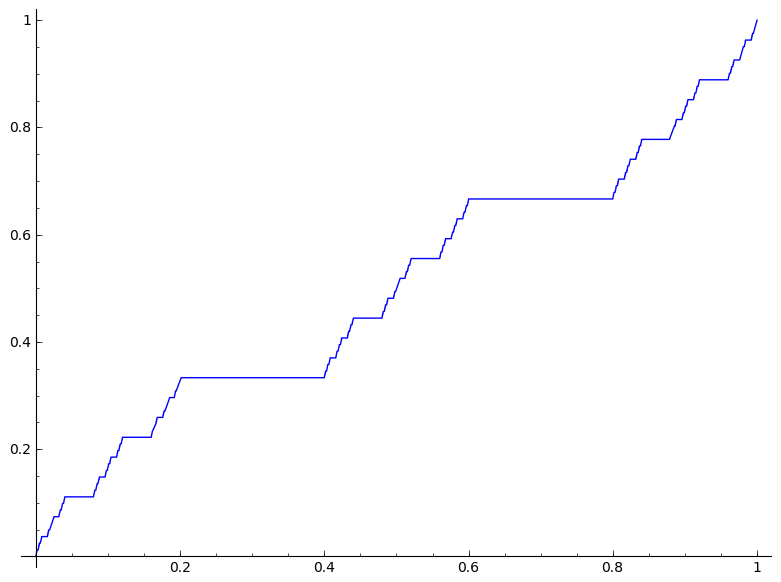}
\end{center}
and of the middle $0.5$-Cantor sets,
\begin{center}
\includegraphics[width=3in]{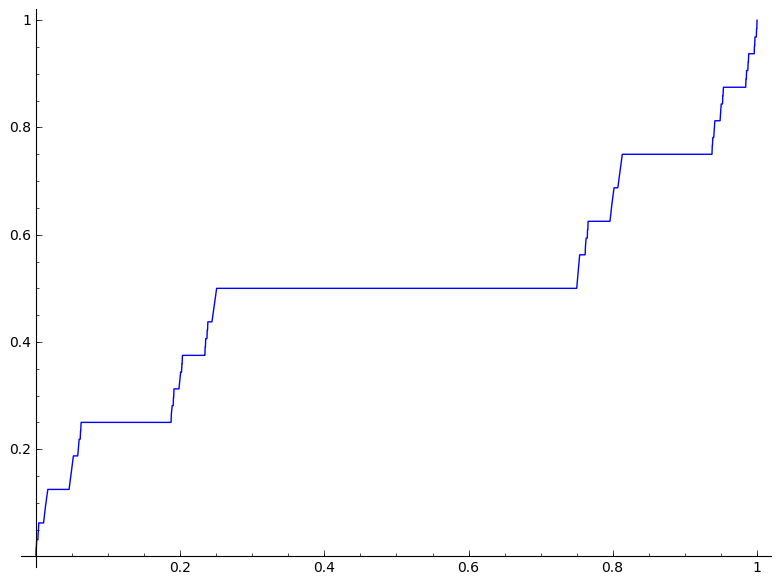}
\end{center}

On the other hand, an analogous construction of the Cantor-Lebesgue function using the ternary expansion can not be done in general. Nevertheless for the case of Cantor-like sets determined by two sequences $\{k_n\}$, $\{c_n\}$ that analogous construction is actually possible, as it is done in \cite{GanVijaVen}. 

Let $E^{k,c},$ be the {\em Cantor-like} set determined by the sequences $\{k_n\}$ and $\{c_n\}$,   a {\em Cantor-Lebesgue function} $G^{k,c} $ can be constructed on $[0,c_0].$ Let $h_n = (k_1\cdot k_2 \cdots k_n)^{-1}$, then the function $G^{k,c}$ is defined as follows: For $\sum_{j=1}^\infty \varepsilon_j r_j \in E$, let
$$G^{k,c}(\sum_{j=1}^\infty \varepsilon_j r_j) = \sum_{j=1}^\infty \varepsilon_j h_j.$$
For each $n \in {\mathbb N}$ , $t = \sum_{j=1}^n \varepsilon_j r_j \in F_n,$ we have
\begin{eqnarray*}
G^{k,c}(t+ c_{n+1}) &=& G^{k,c}(t+ \sum_{j=n+2}^\infty (k_j-1) r_j )= \sum_{j=1}^n \varepsilon_j h_j+ \sum_{j=n+2}^\infty (k_j-1) r_j \\
&=&\sum_{j=1}^n \varepsilon_j h_j +h_{n+1} = G^{k,c}(t+ r_{n+1}).
\end{eqnarray*}

Observe that this corresponds in this general case 
with the property of the Cantor function already mentioned
$$G(t +\frac{1}{3^{n}}) = G(t + \frac{2}{3^{n}})$$

Define $G^{k,c}$ on $(t+c_{n+1}, t+ r_{n+1})$ to have the constant value that it has at both of the endpoints, namely $G^{k,c} (t)+ h_{n+1}.$ Thus we have defined the function $G^{k,c}$ on $[0, c_0]$. 

$G^{k,c}$ is monotone nondecreasing, it is constant in each component subinterval of $[0,c_0] - E^{k,c},$ and it is easy to prove that it is also continuous.\\

\end{document}